\documentclass[a4paper,10pt]{article}
\usepackage{graphicx}

\usepackage{amsmath}
\usepackage{amssymb}
\usepackage{amsfonts}
\usepackage{amsthm}
\usepackage{mathtools}
\usepackage[binary-units=true]{siunitx}
\usepackage{booktabs}
\usepackage[affil-sl]{authblk}
\usepackage{todonotes}
\usepackage{accents}
\usepackage{bm}

\usepackage{dsfont}
\usepackage{url}
\usepackage{enumerate}
\usepackage[inline]{enumitem}
\usepackage{cite}

\usepackage{xspace}
\usepackage{tikz}
\usetikzlibrary{patterns,patterns.meta}
\usepackage{subfigure}

\usepackage[margin=3cm]{geometry}

\definecolor{darkred}{RGB}{150, 0, 0}
\definecolor{darkgreen}{RGB}{0, 150, 0}
\definecolor{darkblue}{RGB}{0, 0, 150}
\usepackage[breaklinks=true,colorlinks,citecolor=darkgreen,linkcolor=darkred,urlcolor=darkblue,bookmarks=false]{hyperref}

\newcommand{\R}{\mathds{R}}
\newcommand{\Z}{\mathds{Z}}

\newcommand{\lexgeq}{\ensuremath{\geq_{\mathrm{lex}}}}

\renewcommand{\dim}{d}
\newcommand{\boxdom}{D}
\newcommand{\ball}{\mathcal{B}}
\newcommand{\dist}{\mathrm{dist}}
\newcommand{\distval}{\delta}
\newcommand{\distvalalt}{q}
\DeclareMathOperator{\argmax}{argmax}

\newcommand{\minDC}{minDC\xspace}
\newcommand{\minDCs}{minDCs\xspace}

\newcommand{\lb}[1]{\underaccent\bar{#1}}
\newcommand{\ub}[1]{\bar{#1}}

\newcommand{\solver}[1]{\texttt{#1}\xspace}
\newcommand{\scip}{\solver{SCIP}}

\newcommand{\abs}[1]{\lvert #1 \rvert}
\newcommand{\define}{\coloneqq}
\newcommand{\norm}[1]{\| #1 \|_2}
\newcommand{\card}[1]{\abs{#1}}
\DeclareMathOperator{\conv}{conv}

\theoremstyle{plain}
\newtheorem{theorem}{Theorem}[section]
\newtheorem{lemma}[theorem]{Lemma}
\newtheorem{proposition}[theorem]{Proposition}

\newtheorem*{claim*}{Claim}

\newtheorem{observation}[theorem]{Observation}
\theoremstyle{definition}
\newtheorem{remark}[theorem]{Remark}

\begin{document}

\title{A computational comparison of handling distance constraints in MINLP}

\author[1]{Christopher Hojny}
\affil[1]{%
  Eindhoven University of Technology, Department of Mathematics and
  Computer Science, PO Box 513, 5600 MB Eindhoven, The Netherlands,
  c.hojny@tue.nl
}
\author[2]{Leo Liberti}
\affil[2]{%
  LIX CNRS Ecole Polytechnique, Institut Polytechnique de Paris, 91128
  Palaiseau, France, leo.liberti@polytechnique.edu
}

\maketitle

\begin{abstract}
  Minimum distance constraints (\minDCs) appear in many geometric optimization problems.
  They pose major challenges for mixed-integer nonlinear programming (MINLP) due to their reverse-convexity.
  We develop new algorithms for tightening variable bounds in general MINLPs with \minDCs.
  Because many such problems exhibit substantial symmetry, we further discuss an approach for handling rotation symmetries.
  In a computational study, we examine the performance of the various methods and determine the scenarios in which each approach demonstrates superiority.

\end{abstract}

\section{Introduction}

We consider mixed-integer nonlinear programs (MINLP)
\[
  \min_{x \in \R^n}\{f(x) : g(x) \leq 0 \text{ and } x_i \in \Z \text{ for } i \in I\},
\]
where~$f\colon \R^n \to \R$, $g\colon \R^n \to \R^m$, $I \subseteq [n] \define \{1,\dots,n\}$, and some constraints are minimum distance constraints.
A \emph{minimum distance constraint (\minDC)} is a constraint~$\norm{y - z}^2 \geq \distval^2$, where~$y,z \in \R^\dim$ are~$\dim$-dimensional subvectors of~$x$ and~$\distval \geq 0$ is either a variable or real number.
Minimum distance constraints require that two points in~$\dim$-dimensional space have Euclidean distance at least~$\distval$.
They arise in many applications, e.g., the sphere packing problem~\cite{LocatelliRaber2002}, kissing number problem~\cite{KucherenkoEtAl2007}, or obnoxious facility location problem~\cite{KusnetsovSahinidis2025}.
Since \minDCs are highly nonconvex, solving MINLP problems involving such constraints is extremely difficult.
For example, for the circle packing problem, where the goal is to find the largest radius such that~$k$ nonoverlapping circles fit into a unit circle, the MINLP solver \scip only finds a provably optimal solution within two hours when~$k \leq 8$.

For many applications, the effective treatment of \minDCs has been investigated.
Most prominently, many variants of the circle packing problem have been studied.
While many heuristic solution approaches have been suggested, see, e.g., \cite{LaiEtAl2022} for an excellent overview for the problem of packing circles into a bigger circle, few exact methods exist for circle packing.
One of the first such methods is~\cite{LocatelliRaber2002}, which suggests a branch-and-bound (BB) algorithm where the variables are the center points of the circles.
If~$\boxdom_y, \boxdom_z \subseteq \R^2$ are bounding boxes for the center points~$y$ and~$z$ in a \minDC, it was observed that the box~$\boxdom_y$ can be replaced by~$\boxdom_y \setminus \hat{\boxdom} \subseteq \boxdom_y$ if the points in~$\hat{\boxdom}$ are ``too close'' to~$\boxdom_z$.
This observation is used to tighten variable bounds and thus accelerate the BB algorithm.
Since this approach is vulnerable to numerical errors, \cite{MarkotCsendes2006} adapts the ideas of~\cite{LocatelliRaber2002} to interval arithmetic, in order to find guaranteed feasible solutions.
Extending on the ideas of~\cite{LocatelliRaber2002}, \cite{pecsdam,Costa2013} derive structural properties of solutions of circle and point packing problems which enable a more effective pruning within BB.
In~\cite{WangGounaris2021}, an alternative BB algorithm is suggested, where branching is not performed on the~$y$- and~$z$-variables, but on the differences~$\Delta_i = y_i - z_i$, $i \in \{1,2\}$.
Exploiting bounds on~$\Delta_i$, intersection cuts and three feasibility-based bound tightening (FBBT) techniques are derived.

For general MINLPs with \minDCs, however, the literature is scarcer and the previously mentioned results, except for the bounding box idea of~\cite{LocatelliRaber2002}, can not be applied.
However, to the best of our knowledge, the effectiveness of this idea has not been investigated for general MINLP problems.
Moreover, to handle \minDCs in general MINLPs, the authors of~\cite{KusnetsovSahinidis2025} considered the case of \minDCs~$\norm{y - z}^2 \geq \distval^2$ where one vector, say~$z$, is a fixed point~$p$.
When there are multiple \minDCs~$\norm{y - p^1}^2 \geq \distval^2_1$, \dots, $\norm{y - p^k}^2 \geq \distval^2_k$ and~$y$ is contained in a bounding box~$\boxdom \subseteq \R^2$, they showed how to efficiently compute the convex hull of all~$y \in \boxdom$ satisfying the~$k$ \minDCs and exploit facet defining inequalities within spatial branch-and-bound.

In this article, we complement these results by considering the case when both~$y$ and~$z$ are variable vectors.
Our aim is to devise new algorithms for handling \minDCs and to computationally investigate scenarios in which the different methods demonstrate superiority.
To this end, Sec.~\ref{sec:mindc} explains the method given in~\cite{LocatelliRaber2002} in more detail and suggests new methods for handling \minDCs.
Since many problems involving \minDCs also admit many symmetries, it is inevitable to handle these symmetries in branch-and-bound because otherwise the performance of MINLP solvers deteriorates drastically, cf.~\cite{Margot2010,PfetschRehn2019}.
While many sophisticated symmetry handling methods exist, for \emph{rotation symmetries} arising in applications like the kissing number problem or packing spheres into a bigger sphere, only simple techniques have been developed to the best of our knowledge.
We show that, despite their simplicity, these techniques are the strongest possible ones for the class of Givens rotations (Sec.~\ref{sec:sym}).
Sec.~\ref{sec:num} concludes this article by a numerical comparison of the different algorithms.

\paragraph{Notation.}
For a variable~$x$, we denote by~$\lb{x}$ and~$\ub{x}$ lower and upper bounds on~$x$, respectively.
Moreover, for~$r > 0$ and~$p \in \R^\dim$, the set~$\ball_r(p) = \{ x \in \R^\dim : \norm{x - p} < r\}$ denotes the open Euclidean ball of radius~$r$ centered at~$p$.
By~$\partial\ball_{r}(p)$ we denote the boundary of~$\ball_r(p)$.
The convex hull of a set~$S \subseteq \R^\dim$ is denoted by~$\conv(S)$.
For a polyhedron~$P$, we denote by~$V(P)$ its vertex set.

\section{Handling minimum distance constraints}
\label{sec:mindc}

Throughout this section, we consider minimum distance constraints
\begin{equation}
  \label{eq:mindc}
  \sum_{i = 1}^\dim (y_i - z_i)^2 \geq \distval^2,
\end{equation}
where~$y,z \in \R^\dim$ are variables with respective domains~$\boxdom_y, \boxdom_z \subseteq \R^\dim$, and~$\distval$ is a nonnegative variable.
When solving MINLPs via spatial branch-and-bound, variable domains are typically real intervals.
We therefore assume that the domains~$\boxdom_y$ and~$\boxdom_z$ are hyperrectangles, which we refer to as boxes for brevity.

In the following, we explain the ideas of Locatelli and Raber~\cite{LocatelliRaber2002} for shrinking the box domains~$\boxdom_y, \boxdom_z$ based on \minDCs (Section~\ref{sec:locatelli}).
Afterwards, we describe a simpler algorithm for shrinking box domains (Section~\ref{sec:algOne}), present simple cutting planes (Section~\ref{sec:algOneCut}), and develop novel algorithms for shrinking box domains based on pairs of \minDCs (Section~\ref{sec:algTwo}).

\subsection{A single minimum distance constraint}
\subsubsection{Locatelli and Raber's algorithm~\cite{LocatelliRaber2002}}
\label{sec:locatelli}

The approach of~\cite{LocatelliRaber2002} proceeds in two steps.
\begin{enumerate}
\item Weaken the \minDC~\eqref{eq:mindc} to~$\sum_{i = 1}^\dim (y_i - z_i)^2 \geq \lb{\distval}^2$, i.e., replace the variable right-hand side~$\distval^2$ by the constant lower bound~$\lb{\distval}^2$.
\item Find a region~$\hat{\boxdom} \subseteq \boxdom_y$ such that, for every~$y \in \hat{\boxdom}$, one has~$\max_{z \in \boxdom_z} \norm{y - z}^2 < \lb{\distval}^2$.
\end{enumerate}
Clearly, every~$(y,z) \in \hat{\boxdom} \times \boxdom_z$ violates the weakened constraint~\eqref{eq:mindc} (i.e., with~$\distval$ replaced by~$\lb{\distval}$), and therefore~$\boxdom_y$ can be replaced by~$\boxdom_y \setminus \hat{\boxdom}$.
The crucial step of this procedure is Step~2.
Although Locatelli and Raber discuss it only for~$\dim = 2$, their results immediately apply to the general case~$\dim \geq 2$.
Their ideas can be summarized as follows.

For every~$y \in \boxdom_y$, note that~$\max_{z \in \boxdom_z} \norm{y - z}^2 = \max_{z \in V(\boxdom_z)} \norm{y - z}^2$ because~$\norm{y - z}^2$ is convex in~$z$.
A point~$y \in \boxdom_y$ can thus only be removed from~$\boxdom_y$ if its Euclidean distance to all~$z \in V(\boxdom_z)$ is less than~$\lb{\distval}$, i.e., $y \in \bigcap_{z \in V(\boxdom_z)} \ball_{\lb{\distval}}(z) \eqqcolon C(\boxdom_z, \lb{\delta})$.
Consequently, the largest region~$\hat{\boxdom}$ that can be removed from~$\boxdom_y$ is
\[
  \hat{\boxdom} = \boxdom_y \cap C(\boxdom_z, \lb{\delta}),
\]
see Fig.~\ref{fig:locatelli} for an illustration, where both~$y$ and~$z$ are points in~$\R^2$ that need to be placed in their respective domains while satisfying a \minDC.

Since~$\boxdom_y \setminus \hat{\boxdom}$ can be nonconvex (the light blue region in Fig.~\ref{fig:locatelli}), Locatelli and Raber suggest to use a smaller set~$\bar{\boxdom}$, which is constructed as follows.
\begin{enumerate}[label=(\roman*)]
\item For every~$z \in V(\boxdom_z)$, one computes the set~$J_z$ consisting of all intersection points of~$\ball_{\lb{\distval}}(z)$ with the edges of~$\boxdom_y$.
  Let~$J = V(\boxdom_y) \cup \bigcup_{z \in V(\boxdom_z)} J_z$
\item Select~$\bar{\boxdom} = \conv(\bar{J})$, where~$\bar{J} = \{x \in J : \max_{z \in V(\boxdom_z)} \norm{x - z}^2 < \lb{\distval}^2\}$ contains all vertices of~$\boxdom_y$ and intersection points of~$C(\boxdom_z, \lb{\delta})$ with the edges of~$\boxdom_y$ that are at distance less than~$\lb{\distval}$ to each vertex of~$\boxdom_z$.
\end{enumerate}
Although this choice guarantees that the set~$\boxdom_y \setminus \bar{\boxdom}$ is convex and remains polyhedral, it has two downsides: maintaining an exact representation of~$\boxdom_y \setminus \bar{\boxdom}$ potentially requires a lot of memory, and computing a facet description of~$\boxdom_y \setminus \bar{\boxdom}$ is costly in arbitrary dimension.
Locatelli and Raber therefore suggest replacing~$\bar{\boxdom}$ by a box~$\boxdom'$ that has a common facet with~$\boxdom_y$, see Fig.~\ref{fig:locatelliB}.
This approach ensures that~$\boxdom_y \setminus \boxdom'$ remains a box.
\begin{observation}
  Since~$\card{V(\boxdom_y)}, \card{V(\boxdom_z)} \in O(2^\dim)$ and a~$\dim$-dimensional box has~$\dim 2^{\dim - 1}$ edges, the set~$J$ has size~$\card{J} \in O(\card{V(\boxdom_z)} \cdot \dim 2^{\dim - 1} + \card{V(\boxdom_y)}) = O(\dim 2^{2\dim - 1})$.
  Checking containment of~$x \in J$ in~$\bar{J}$ takes~$O(\dim \card{V(\boxdom_z)})$ time because computing~$\norm{x - z}$ requires~$O(\dim)$ time for a fixed~$z \in \boxdom_z$.
  Consequently, the entire procedure of Locatelli and Raber takes $O(\card{J} \cdot \dim \card{V(\boxdom_z)}) =O(\dim^2 2^{3\dim - 1})$ time.
\end{observation}

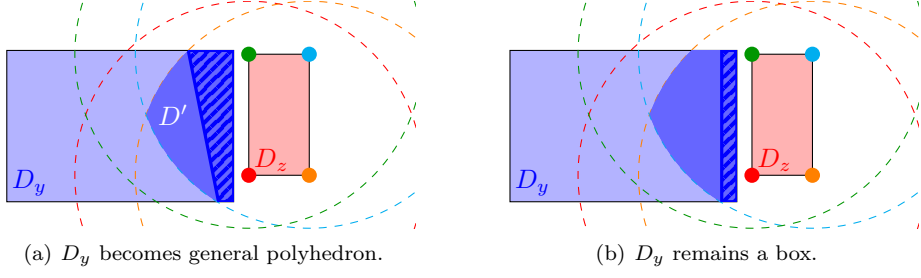
\begin{figure}
  \centering
  \subfigure[$\boxdom_y$ becomes general polyhedron.\label{fig:locatelliA}]{
    \begin{tikzpicture}
      \clip (-0.2,-0.5) rectangle (5.4,2.5);
      \draw[fill=blue!30] (0,-0.15) rectangle (3,1.85);
      \draw[fill=red!30] (3.2,0.2) rectangle (4.0,1.8);
      \node (Y) at (0.3,0.1) {\textcolor{blue}{$\boxdom_y$}};
      \node (Z) at (3.5,0.4) {\textcolor{red}{$\boxdom_z$}};

      \node (R1) at (3.2,0.2) [circle,fill=red,minimum size=2mm,inner sep=0mm] {};
      \node (R2) at (4.0,0.2) [circle,fill=orange,minimum size=2mm,inner sep=0mm] {};
      \node (R3) at (4.0,1.8) [circle,fill=cyan,minimum size=2mm,inner sep=0mm] {};
      \node (R4) at (3.2,1.8) [circle,fill=darkgreen,minimum size=2mm,inner sep=0mm] {};

      \draw[dashed,draw=red] (R1) circle (2.3cm);
      \draw[dashed,draw=orange] (R2) circle (2.3cm);
      \draw[dashed,draw=cyan] (R3) circle (2.3cm);
      \draw[dashed,draw=darkgreen] (R4) circle (2.3cm);
      \clip (0,-0.15) rectangle (3,1.85);
      \clip (R2) circle (2.3cm);
      \clip (R3) circle (2.3cm);
      \draw[dashed,draw=orange,fill=blue!60] (R2) circle (2.3cm);
      \draw[dashed,draw=cyan,fill=blue!60] (R3) circle (2.3cm);
      \node (DD) at (2.2,1) {\textcolor{white}{$\hat{D}$}};

      \draw[-,very thick,draw=blue,pattern={Lines[angle=45,line width=0.5mm]},pattern color=blue] (2.4,1.85) -- (2.8,-0.15) -- (3,-0.15) -- (3,1.85) -- cycle;
      \node (DD) at (2.75,1.5) {\textcolor{white}{$\bar{D}$}};
    \end{tikzpicture}
  }
  \qquad
  \subfigure[$\boxdom_y$ remains a box.\label{fig:locatelliB}]{
    \begin{tikzpicture}
      \clip (-0.2,-0.5) rectangle (5.4,2.5);
      \draw[fill=blue!30] (0,-0.15) rectangle (3,1.85);
      \draw[fill=red!30] (3.2,0.2) rectangle (4.0,1.8);
      \node (Y) at (0.3,0.1) {\textcolor{blue}{$\boxdom_y$}};
      \node (Z) at (3.5,0.4) {\textcolor{red}{$\boxdom_z$}};

      \node (R1) at (3.2,0.2) [circle,fill=red,minimum size=2mm,inner sep=0mm] {};
      \node (R2) at (4.0,0.2) [circle,fill=orange,minimum size=2mm,inner sep=0mm] {};
      \node (R3) at (4.0,1.8) [circle,fill=cyan,minimum size=2mm,inner sep=0mm] {};
      \node (R4) at (3.2,1.8) [circle,fill=darkgreen,minimum size=2mm,inner sep=0mm] {};

      \draw[dashed,draw=red] (R1) circle (2.3cm);
      \draw[dashed,draw=orange] (R2) circle (2.3cm);
      \draw[dashed,draw=cyan] (R3) circle (2.3cm);
      \draw[dashed,draw=darkgreen] (R4) circle (2.3cm);
      \clip (0,-0.15) rectangle (3,1.85);
      \clip (R2) circle (2.3cm);
      \clip (R3) circle (2.3cm);
      \draw[dashed,draw=orange,fill=blue!60] (R2) circle (2.3cm);
      \draw[dashed,draw=cyan,fill=blue!60] (R3) circle (2.3cm);

      \draw[-,very thick,draw=blue,pattern={Lines[angle=45,line width=0.5mm]},pattern color=blue] (2.8,1.85) -- (2.8,-0.15) -- (3,-0.15) -- (3,1.85) -- cycle;
      \node (DD) at (2.82,1.5) {$\textcolor{white}{D'}$};
    \end{tikzpicture}
  }
  \caption{Illustration of the method by Locatelli and Raber to generate convex domains~$\boxdom_y$.
    In Figs.~\ref{fig:locatelliA} and \ref{fig:locatelliB}, the hatched area corresponds to~$\bar{\boxdom}$ and~$\boxdom'$, respectively.
  }\label{fig:locatelli}
\end{figure}

\subsubsection{A simple algorithm for shrinking box domains}
\label{sec:algOne}

Since the running time of the algorithm by Locatelli and Raber grows exponentially in~$\dim$, it is only practical for moderate values of~$\dim$.
We therefore present a much simpler algorithm that scales linearly in the dimension~$\dim$.
While the algorithm of Locatelli and Raber provably finds the largest box~$\boxdom'$ that can be removed from~$\boxdom_y$ such that~$\boxdom_y \setminus \boxdom'$ remains a box, our algorithm comes without such a guarantee.

For all~$i \in [\dim]$, let~$\dist_i = \max_{(y,z) \in \boxdom_y \times \boxdom_z} |y_i - z_i|$.
For all~$j \in [\dim]$, let~$\distvalalt_j = \lb{\distval}^2 - \sum_{i \in [d] \setminus \{j\}} \dist_i^2$.
\begin{lemma}
  \label{lem:simpleMindc}
  With respect to Eq.~\eqref{eq:mindc}, we have~$(y_j - z_j)^2 \geq \distvalalt_j$ for every~$j \in [\dim]$.
\end{lemma}
\begin{proof}
  For any dimension index~$j \in [\dim]$, consider Eq.~\eqref{eq:mindc} in the following form:
  \[
    \sum_{i \in [d] \setminus \{j\}} (y_i - z_i)^2 + (y_j - z_j)^2 \geq \distval^2.
  \]
  This yields~$(y_j - z_j)^2 \geq \distval^2 - \sum_{i \in [d] \setminus \{j\}} (y_i - z_i)^2$.
  Now, we have~$\lb{\distval}^2 \leq \distval^2$, as well as~$(y_i - z_i)^2 \leq \dist_i^2$, which implies that
  \[
    \distval^2 - \sum_{i \in [d] \setminus \{j\}} (y_i - z_i)^2
    \geq
    \lb{\distval}^2 - \sum_{i \in [d] \setminus \{j\}} \dist_i^2
    =
    \distvalalt_j,
  \]
  as claimed.
\end{proof}
Based on this observation, we can devise some simple rules for reducing
variable domains, see Fig.~\ref{fig:illustratePropFast} for an illustration.
\begin{proposition}
  \label{prop:prop1fast}
  Consider a single \minDC~\eqref{eq:mindc} with
  variable domains~$\boxdom_y$ and~$\boxdom_z$, and let~$j \in [n]$.
  Let~$(y',z') \in \boxdom_y \times \boxdom_z$ satisfy~\eqref{eq:mindc}.
  If~$\distvalalt_j > 0$, then
  \begin{enumerate}
  \item if~$\ub{y}_j \leq \lb{z}_j$, we have
    \begin{align*}
      y'_j &\leq \min\{\ub{y}_j, \ub{z}_j - \sqrt{\distvalalt_j}\},\\
      z'_j &\geq \max\{\lb{z}_j, \lb{y}_j + \sqrt{\distvalalt_j}\},
    \end{align*}
    and symmetrically if~$\ub{z}_j \leq \lb{y}_j$;
  \item if~$\lb{z}_j \leq \lb{y}_j \leq \ub{y}_j \leq \ub{z}_j$, then~$\ub{y}_j - \lb{z}_j \geq \sqrt{\distvalalt_j}$ or~$\ub{z}_j - \lb{y}_j \geq \sqrt{\distvalalt_j}$,
    and symmetrically if~$\lb{y}_j \leq \lb{z}_j \leq \ub{z}_j \leq \ub{y}_j$;
  \item if~$\lb{z}_j < \lb{y}_j \leq \ub{z}_j < \ub{y}_j$, we have
    \begin{align*}
      \ub{z}_j - \lb{y}_j < \sqrt{\distvalalt_j}
      \text{ and }
      \ub{y}_j - \ub{z}_j < \sqrt{\distvalalt_j}
      \implies&
                z'_j \leq \min\{\ub{z}_j, \ub{y}_j - \sqrt{\distvalalt_j}\};\\
      \lb{y}_j - \lb{z}_j < \sqrt{\distvalalt_j}
      \text{ and }
      \ub{z}_j - \lb{y}_j < \sqrt{\distvalalt_j}
      \implies&
                y'_j \geq \max\{\lb{y}_j, \lb{z}_j + \sqrt{\distvalalt_j}\},
    \end{align*}
    and symmetrically if~$\lb{y}_j < \lb{z}_j \leq \ub{y}_j < \ub{z}_j$.
  \end{enumerate}
\end{proposition}
\begin{figure}[t]
  \centering
  \subfigure[Case 1: dashed parts can be removed from the domains.\label{fig:illus1}]{
    \begin{tikzpicture}
      \clip (-0.75,-1) rectangle (3.25,1);
      \draw[very thick,red,dashed] (0,0) -- (1,0);
      \draw[very thick,blue,dashed] (1.5,0) -- (2.5,0);
      \draw[very thick,red] (0,0) -- (0.5,0);
      \draw[very thick,blue] (2,0) -- (2.5,0);
      \node (R1) at (0,0) [circle,fill=red,minimum size=1mm,inner sep=0mm] {};
      \node (B1) at (2.5,0) [circle,fill=blue,minimum size=1mm,inner sep=0mm] {};
      \draw[dashed,red] (R1) circle (2cm);
      \draw[dashed,blue] (B1) circle (2cm);
    \end{tikzpicture}
  }
  \quad
  \subfigure[Case 2: the blue interval is not too close to any of the
  endpoints of the red interval.\label{fig:illus2}]{
    \begin{tikzpicture}
      \clip (-0.75,-1) rectangle (3.75,1);
      \draw[very thick,blue] (0,0) -- (3,0);
      \draw[ultra thick,red] (0.5,0.05) -- (2.5,0.05);
      \node (R1) at (0,0) [circle,fill=blue,minimum size=1mm,inner sep=0mm] {};
      \node (R2) at (3,0) [circle,fill=blue,minimum size=1mm,inner sep=0mm] {};
      \draw[dashed,blue] (R1) circle (2cm);
      \draw[dashed,blue] (R2) circle (2cm);
    \end{tikzpicture}
  }
  \quad
  \subfigure[Case 3: dashed parts can be removed from the domain.\label{fig:illus3}]{
    \begin{tikzpicture}
      \clip (-0.75,-1) rectangle (3.75,1);
      \draw[very thick,blue] (0,0) -- (0.8,0);
      \draw[very thick,dashed,blue] (0.8,0) -- (2,0);
      \draw[ultra thick,red] (1,0.05) -- (3,0.05);
      \node (B1) at (3,0.05) [circle,fill=red,minimum size=1mm,inner sep=0mm] {};
      \draw[dashed,red] (B1) circle (2.2cm);
    \end{tikzpicture}
  }
  \caption{Illustration of the three different cases of
    Prop.~\ref{prop:prop1fast}.
    The red and blue lines correspond to the interval domain of~$y_j$
    and~$z_j$, respectively.
  }
  \label{fig:illustratePropFast}
\end{figure}
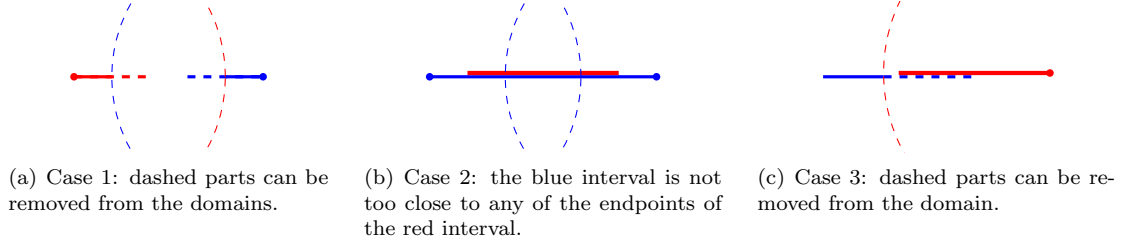
\begin{proof}
  Recall that, since~$\boxdom_y$ and~$\boxdom_z$ are boxes, the domain of variable~$y_j$ and~$z_j$ is an interval for each~$j \in [\dim]$.
  Let~$(y', z') \in \boxdom_y \times \boxdom_z$ adhere to~\eqref{eq:mindc}.
  Then, Lem.~\ref{lem:simpleMindc} implies~$(y'_j - z'_j)^2 \geq \distvalalt_j$.
  For each of the three cases, we discuss one variant as the other one follows by symmetry.

  In the first case, $(y'_j - z'_j)^2 \geq \distvalalt_j$ yields~$z'_j - y'_j \geq \sqrt{\distvalalt_j}$ since~$z'_j \geq y'_j$.
  Hence, $y'_j \leq z'_j - \sqrt{\distvalalt_j} \leq \ub{z}_j - \sqrt{\distvalalt_j}$.
  Together with the trivial inequality~$y'_j \leq \ub{y}_j$, we conclude~$y'_j \leq \min\{\ub{y}_j, \ub{z}_j - \sqrt{\distvalalt_j}\}$.

  In the second case, the domain of~$y_j$ is a subset of the domain of~$z_j$.
  Since~$(y_j - z_j)^2$ is convex, its maximum value is attained at a vertex of the feasible region~$[\lb{y},\ub{y}] \times [\lb{z}, \ub{z}]$.
  That is,
  \begin{align*}
    \max_{y,z} (y_j - z_j)^2
    &=
      \max\{ (\lb{y}_j - \lb{z}_j)^2, (\lb{y}_j - \ub{z}_j)^2, (\ub{y}_j - \lb{z}_j)^2, (\ub{y}_j - \ub{z}_j)^2\}\\
    &=
      \max\{ (\lb{y}_j - \ub{z}_j)^2, (\ub{y}_j - \lb{z}_j)^2\},
  \end{align*}
  where the latter holds because the domain of~$y$ is a subset of the domain of~$z$.
  Consequently, if~$\ub{y}_j - \lb{z}_j < \sqrt{\distvalalt_j}$ and~$\ub{z}_j - \lb{y}_j < \sqrt{\distvalalt_j}$, then~$\max_{y,z} (y_j - z_j)^2 < \distvalalt_j$, contradicting Lem.~\ref{lem:simpleMindc}.
  Thus, $\ub{y}_j - \lb{z}_j \geq \sqrt{\distvalalt_j}$ or~$\ub{z}_j - \lb{y}_j \geq \sqrt{\distvalalt_j}$ holds.

  For the last case, we only discuss the first subcase, because the second one follows analogously.
  By the same arguments as in Case~2, $\ub{z}_j - \lb{y}_j < \sqrt{\distvalalt_j}$ and~$\ub{y}_j - \ub{z}_j < \sqrt{\distvalalt_j}$ implies that every point in~$[\lb{y}_j,\ub{y}_j]$ has distance less than~$\sqrt{\distvalalt_j}$ to~$\ub{z}_j$.
  Hence, $z'_j \leq \ub{y}_j - \sqrt{\distvalalt_j}$ needs to hold.
  Together with the trivial inequality~$z'_j \leq \ub{z}_j$, we obtain~$z'_j \leq \min\{\ub{z}_j, \ub{y}_j - \sqrt{\distvalalt_j}\}$.
\end{proof}
For each~$i \in [\dim]$, the value~$\dist_i$ can be computed in constant time, because~$(x_i - y_i)^2$ is a convex function and its maximum over the domain~$\boxdom = [\lb{y}_i, \ub{y}_i] \times [\lb{z}_i, \ub{z}_i]$ is attained at one of the (at most) four vertices of~$\boxdom$.
The variable bounds of Prop.~\ref{prop:prop1fast} can therefore be
computed in~$O(\dim)$ time.

In our implementation, we receive the domains~$\boxdom_y$ and~$\boxdom_z$ as well as a \minDC $\norm{y - z}^2 \geq \distval^2$ as input and compute~$\dist_j$ for~$j \in [\dim]$ according to these domains.
We then compute~$T = \sum_{j = 1}^\dim\dist^2_j$, iterate over all dimension indices~$j \in [\dim]$, and possibly apply a reduction according to Prop.~\ref{prop:prop1fast} for the 1-dimensional \minDC $(y_j - z_j)^2 \geq \lb{\distval}^2 - (T - \dist^2_j)$.
Since reducing one variable domain might allow to reduce domains of other variables, cf.\ the wave effect in~\cite{LocatelliRaber2002}, possibly more reductions can be found by calling the method multiple times.
We therefore use this method iteratively until no domain reductions are found anymore or an iteration limit is hit, see Sec.~\ref{sec:num} for details.
\subsubsection{Simple cutting planes}
\label{sec:algOneCut}
The algorithms presented in the previous section can only shrink the domain~$\boxdom_y$ if there is a facet of~$\boxdom_y$ that is contained in~$\ball_{\lb{\distval}}(z)$ for every~$z \in V(\boxdom_z)$.
In the situation of Fig.~\ref{fig:cut}, one therefore needs to use cutting planes to cut off parts of~$\boxdom_y \cap C(\boxdom_z, \lb{\distval})$.
Since the method by Locatelli and Raber for constructing the convex hull of~$\boxdom_y \cap C(\boxdom_z, \lb{\distval})$ is rather expensive, we suggest a simpler set of cutting planes.

\begin{figure}
  \centering

  \subfigure[Cutting plane for one \minDC.\label{fig:cut}]{
    \begin{tikzpicture}
      \clip (-0.2,-1.2) rectangle (6,3.2);
      \draw[fill=blue!30] (1,0) rectangle (5.2,2);
      \node (Dy) at (1.5,1.5) {$\textcolor{blue}{D_y}$};
      \draw[fill=red!30] (4.2,-0.2) rectangle (5.5,-1);
      \node (Dz) at (4.6,-0.5) {$\textcolor{red}{D_z}$};
      \node (R1) at (5.5,-1) [circle,fill=red,minimum size=2mm,inner sep=0mm] {};
      \node (R2) at (5.5,-0.2) [circle,fill=darkgreen,minimum size=2mm,inner sep=0mm] {};
      \node (R3) at (4.2,-1) [circle,fill=cyan,minimum size=2mm,inner sep=0mm] {};
      \node (R4) at (4.2,-0.2) [circle,fill=orange,minimum size=2mm,inner sep=0mm] {};

      \draw[dashed,draw=red] (R1) circle (2.5cm);
      \draw[dashed,draw=darkgreen] (R2) circle (2.5cm);
      \draw[dashed,draw=cyan] (R3) circle (2.5cm);
      \draw[dashed,draw=orange] (R4) circle (2.5cm);

      \clip (3.2,0) rectangle (5.2,1.5);
      \clip (R1) circle (2.5cm);
      \clip (R3) circle (2.5cm);
      \draw[dashed,draw=red,fill=blue!60] (R1) circle (2.5cm);
      \draw[dashed,draw=cyan,fill=blue!60] (R3) circle (2.5cm);
      \draw[-,very thick,draw=blue,pattern={Lines[angle=45,line width=0.5mm]},pattern color=blue] (3.27,0.01) -- (5.185,0) -- (5.185,1.25) -- cycle;
    \end{tikzpicture}
  }
  \quad\quad
  \subfigure[Shrinking box domains for two \minDCs.\label{fig:twoConss}]{
    \begin{tikzpicture}
      \clip (-0.2,-1.2) rectangle (5.4,3.2);
      \draw[fill=blue!30] (0,0) rectangle (4,2);
      \node (Dy) at (0.5,1.5) {$\textcolor{blue}{D_y}$};
      \draw[fill=red!30] (4.2,-0.2) rectangle (5,-1);
      \node (Dz) at (4.6,-0.6) {$\textcolor{red}{D_{z^{(1)}}}$};
      \draw[fill=darkgreen!30] (4.4,2.2) rectangle (5.2,3);
      \node (Dz) at (4.8,2.6) {$\textcolor{darkgreen}{D_{z^{(2)}}}$};
      \node (G) at (5.2,3) [circle,fill=darkgreen,minimum size=2mm,inner sep=0mm] {};
      \node (R) at (5,-1) [circle,fill=red,minimum size=2mm,inner sep=0mm] {};

      \draw[dashed] (R) circle (2.5cm);
      \draw[dashed] (G) circle (2.5cm);
      \clip (2.7,0) rectangle (4,2);
      \draw[dashed,fill=blue!60] (R) circle (2.5cm);
      \clip (2.8,0) rectangle (4,2);
      \draw[dashed,fill=blue!60] (G) circle (2.5cm);
      \draw[dashed] (R) circle (2.5cm);
      \draw[-,very thick,draw=blue,pattern={Lines[angle=45,line width=0.5mm]},pattern color=blue] (3.6,0.015) -- (3.985,0.015) -- (3.985,1.985) -- (3.6,1.985) -- cycle;
      \node (DD) at (3.8,0.2) {$\textcolor{white}{D'}$};
    \end{tikzpicture}
  }  
  \caption{Illustration of domain reductions derived via cutting planes and two \minDCs.}\label{fig:2Ddistances}
\end{figure}
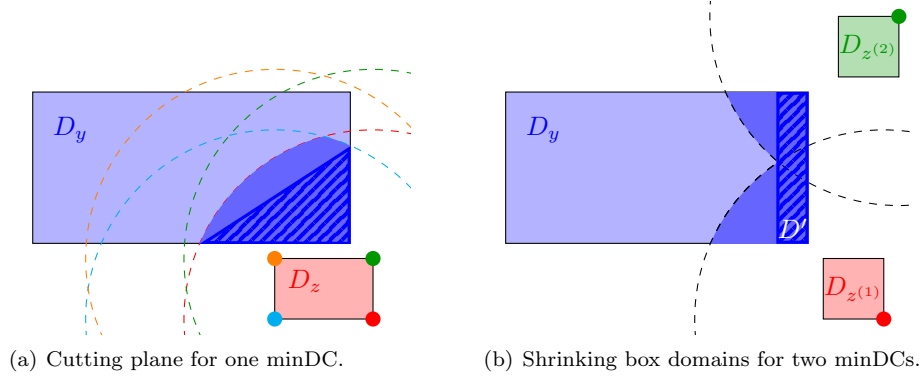

Let us assume that~$\boxdom_y$ is full-dimensional and~$y' \in V(\boxdom_y) \cap C(\boxdom_z, \lb{\distval})$.
To find a cutting plane that cuts off~$y'$, let~$E$ be the set of~$\dim$ (polyhedral) edges of~$\boxdom_y$ that are incident with~$y'$.
For~$e \in E$, let~$p(e) = \argmax\{\norm{y' - p} : p \in e \cap C(\boxdom_z, \lb{\distval})\}$, i.e., $p(e)$ is the point on~$e$ that is farthest away from~$y'$ while still contained in~$C(\boxdom_z, \lb{\distval})$.
Since~$y' \in C(\boxdom_z, \lb{\distval})$, the simplex~$\conv(\{y'\} \cup \{p(e) : e \in E\})$ is contained in~$C(\boxdom_z, \lb{\distval})$ and the facet defined by~$\{p(e) : e \in E\}$ can be used as a cutting plane to cut off~$y'$, see Fig.~\ref{fig:cut}.

In our implementation, we derive these inequalities from the box domain of a variable~$y$.
That is, even when cutting planes are found, they are not taken into account for deriving possibly stronger cuts in subsequent iterations.
This allows to keep a simple representation of the domain of~$y$, which is then a box rather than a general polytope, cf.\ the discussion of using~$\boxdom'$ instead of~$\bar{\boxdom}$ in Sec.~\ref{sec:locatelli} and~\cite{LocatelliRaber2002}.
\begin{observation}
  For fixed~$y' \in V(\boxdom_y)$ and~$e \in E$, $p(e)$ can be found by computing the intersection points of~$e$ and~$\partial\ball_{\lb{\distval}}(z)$ for every~$z \in V(\boxdom_z)$.
  Computing the intersections of~$e$ with a single sphere involves~$O(\dim)$ numbers; a single intersection can thus be computed in~$O(\dim)$ time.
  As~$\boxdom_z$ has~$O(2^\dim)$ vertices, the set~$\{p(e) : e \in E\}$ can be found in~$O(\dim2^{\dim})$ time.
  Finding the hyperplane spanned by~$\{p(e) : e \in E\}$ requires the solution of a system of~$\dim$ linear equations.
  The time for computing a single hyperplane is thus dominated by the time to find~$\{p(e) : e \in E\}$ and is therefore~$O(\dim2^{\dim})$.
\end{observation}
\begin{remark}
  The cutting planes described above share similarities with intersection cuts~\cite{Balas1971,ConfortiEtAl2014}.
  Intersection cuts are cutting planes that are derived from the intersection of a polyhedron (in our case~$\boxdom_y$) and a convex set (in our case~$C(\boxdom_z, \lb{\distval})$) that does not contain a feasible solution.
  While our inequalities are used to cut off points that violate a \minDC, intersection cuts are used in mixed-integer programming to remove parts of a polyhedron that do not contain integer points.
\end{remark}
\subsection{Handling pairs of minimum distance constraints}
\label{sec:algTwo}
Consider the situation of Fig.~\ref{fig:twoConss}, where we are given two \minDCs
\begin{align}
  \label{eq:twoDistConss}
  \norm{y - z^{(1)}}^2 &\geq \distval^2_1 &\text{and} && \norm{y - z^{(2)}}^2 &\geq \distval^2_2
\end{align}
that constrain a common variable vector~$y \in \R^\dim$.
That is, both constraints can use possibly different minimum distance values.
In this example, the method of Locatelli and Raber fails to find a smaller box domain for~$y$ because every facet of~$\boxdom_y$ contains a point that is at distance at least~$\distval_1$ (resp.~$\distval_2$) to a point in~$\boxdom_{z^{(1)}}$ (resp.~$\boxdom_{z^{(2)}}$).
By considering both \minDCs simultaneously, however, the hatched area~$\boxdom'$ can be removed from~$\boxdom_y$ because every point in~$\boxdom'$ has distance below~$\distval_1$ to a point in~$\boxdom_{z^{(1)}}$ or distance below~$\distval_2$ to a point in~$\boxdom_{z^{(2)}}$.
In the following, we show how to systematically exploit this observation.
To this end, we proceed in two steps.
First, we discuss how to decide whether a given box domain~$\boxdom'$ can be removed from~$\boxdom_y$ when we are given two \minDCs as in~\eqref{eq:twoDistConss}.
Second, we show how to (heuristically) find candidates for such domains~$\boxdom'$.

\paragraph{Deciding whether~$\bm{\boxdom'}$ can be removed.}
Suppose we have guessed a box domain~$\boxdom'$ that we want to remove.
Then, the following proposition shows how to verify whether~\eqref{eq:twoDistConss} implies~$y \in \boxdom_y \setminus \boxdom'$ in case both~$\boxdom_{z^{1}}$ and~$\boxdom_{z^{2}}$ are singletons.
\begin{proposition}
  \label{prop:boxpairs}
  Let~$\boxdom_y, \boxdom' \subseteq \R^\dim$ be boxes, let~$\boxdom_{z^{(1)}} = \{p\}, \boxdom_{z^{2}} = \{q\} \subseteq \R^\dim$, and let~$\distval_1, \distval_2 > 0$.
  Then, $\boxdom' \subseteq \ball_{\distval_1}(p) \cup \ball_{\distval_2}(q)$ if and only if
  \begin{itemize}
  \item the vertices of~$\boxdom'$ are contained in~$\ball_{\distval_1}(p) \cup \ball_{\distval_2}(q)$, and
  \item for every edge~$e$ of~$\boxdom'$, we have~$e \cap \partial \ball_{\distval_1}(p) \subseteq \ball_{\distval_2}(q)$.
  \end{itemize}
\end{proposition}
\begin{proof}
  First, assume~$\boxdom' \subseteq \ball_{\distval_1}(p) \cup \ball_{\distval_2}(q) \eqqcolon B$.
  Then, the vertices of~$\boxdom'$ are obviously contained in~$B$.
  Moreover, $B$ can only cover~$\boxdom'$ when~$\ball_{\distval_2}(q)$ contains~$\boxdom' \cap \partial \ball_{\distval_1}(p)$.

  Second, assume that the two properties hold.
  Let~$V_1$ and~$V_2$ be the vertices of~$\boxdom'$ that are contained
  in~$\ball_{\distval_1}(p)$ and~$\ball_{\distval_2}(q)$, respectively.
  Moreover, let~$I$ be the intersection points of~$\partial \ball_{\distval_1}(p)$ and the
  edges of~$\boxdom'$.
  We need to show that~$\ball_{\distval_2}(q)$ contains~$A \define \boxdom' \setminus \ball_{\distval_1}(p)$.
  Since~$A$ is a reverse convex set, $A \subseteq \conv(I \cup V_2)$.
  Thus, $A \subseteq \ball_{\distval_2}(q)$ because~$\conv(I \cup V_2) \subseteq \ball_{\distval_2}(q)$ due to the convexity of~$\ball_{\distval_2}(q)$.
\end{proof}
When~$\boxdom_{z^{(1)}}$ and~$\boxdom_{z^{(2)}}$ are not singletons, then the same arguments as in Sec.~\ref{sec:locatelli} can be used to show that~$\boxdom'$ can be removed from~$\boxdom_y$:
for every pair of vertices~$(p,q) \in V(\boxdom_{z^{(1)}}) \times V(\boxdom_{z^{(2)}})$, one needs to check whether Prop.~\ref{prop:boxpairs} applies.
\begin{observation}
  For fixed~$p \in V(\boxdom_{z^{(1)}})$ and~$q \in V(\boxdom_{z^{(2)}})$, checking~$V(\boxdom') \subseteq \ball_{\distval_1}(p) \cup \ball_{\distval_2}(q)$ takes~$O(\dim 2^{\dim})$ time, because~$\card{V(\boxdom')} \in O(2^\dim)$ and computing the distance between two points requires linear time.
  Finding the intersection of~$\partial\ball_{\distval_1}(p)$ with an edge of~$\boxdom'$ and computing the distance to~$q$ takes~$O(\dim)$ time.
  The entire procedure for fixed~$p$ and~$q$ thus takes~$O(\dim 2^{\dim} + \dim^2 2^{\dim - 1}) = O(\dim^2 2^\dim)$ time, because~$\boxdom'$ has~$O(\dim 2^{\dim - 1})$ edges.
  Running this algorithm for all~$O(2^{2 \dim})$ pairs~$(p,q)$ of vertices therefore takes~$O(\dim^2 2^{3\dim})$ time.
\end{observation}

\paragraph{Finding~$\bm{\boxdom'}$.}
To find a candidate box~$\boxdom'$ to remove, we propose two heuristic approaches.
Both approaches iterate through~$i \in [\dim]$ and check whether the domain of~$y_i$ can be reduced.
Since removing a box~$\boxdom'$ from~$\boxdom_y$ results in changing a variable lower or upper bound on~$y_i$, the first approach selects the new bound via bisection on the interval~$[\lb{y}_i, \ub{y}_i]$ and checks whether the removed box satisfies the properties of Prop.~\ref{prop:boxpairs}.
If not, the size of the removed box is decreased; otherwise, the size of the removed box is increased.
Since the time needed to check the requirements of Prop.~\ref{prop:boxpairs} is rather high, we restrict the number of iterations in the bisection to three in our implementation.
Moreover, to increase the chance of finding a reduction, we split the interval with a ratio of~1:9, where~$\boxdom'$ gets~\SI{10}{\percent} of the original interval length.
If no reduction can be found within the iteration limit, the algorithm continues with the next variable without changing the domain of~$y_i$.

The second approach is restricted to the case~$\dim \in \{2,3\}$.
It selects a coordinate~$y_i$, $i \in [\dim]$, and uses geometric ideas to improve a variable bound, say~$\ub{y}_i$.
Motivated by Fig.~\ref{fig:twoConss}, the improved upper bound for~$y_i$ for~$\dim = 2$ is defined by the intersection of two spheres or the intersection of a sphere and an edge of~$\boxdom_y$.
Let~$E$ be the edges of~$\boxdom_y$ that are orthogonal to the facet defined by~$y_i = \ub{y}_i$.
For each~$p \in V(\boxdom_{z^{(1)}})$, $q \in V(\boxdom_{z^{(2)}})$, and~$e \in E$, we collect the intersection points of~$e$ with~$\partial\ball_{\lb{\distval}_1}(p)$ and~$\partial\ball_{\lb{\distval}_2}(q)$.
Let~$I$ be the set of all such intersection points.
Moreover, for~$\dim = 2$, compute the intersections of the spheres~$\partial\ball_{\lb{\distval}_1}(p)$ and~$\partial\ball_{\lb{\distval}_2}(q)$.
Let~$J_2$ be the set of all intersections for which the~$i$-th coordinate lies in the interval~$[\lb{y}_i, \ub{y}_i]$.
The guess for the box~$\boxdom'$ is then given by reducing the upper bound~$\ub{y}_i$ to~$\max\{ y_i : y \in I \cup J_2\}$ and keeping all other bounds identical.

When~$\dim = 3$, we also compute the set~$I$.
But since the intersection of two spheres can be a circle, we check the heights~$y_i$ at which these circles intersect the facets of~$\boxdom_y$ that are incident with the facet defined by~$y_i = \ub{y}_i$.
Let~$J_3$ be the set of all intersections for which the~$i$-th coordinate lies in the interval~$[\lb{y}_i, \ub{y}_i]$.
Then, as for~$\dim = 2$, our guess for the box~$\boxdom'$ is given by reducing the upper bound~$\ub{y}_i$ to~$\max\{ y_i : y \in I \cup J_3\}$ and keeping all other bounds identical.
\section{Handling rotation symmetries}
\label{sec:sym}
Many geometric optimization problems require finding an optimal arrangement of~$n$ points in~$\R^\dim$.
Examples include sphere packing and kissing number problems, where the point coordinates can be represented by a matrix~$X \in \R^{n \times \dim}$.
These problems exhibit substantial symmetry: exchanging identical spheres yields equivalent solutions, while additional symmetries arise from permuting identical dimensions or applying rotations to the coordinates of all points.
In the following, we denote the~$i$-th row of~$X$ by~$X^i$.

When solving these problems with spatial BB, such symmetries must be addressed; otherwise, symmetric regions of the search space are explored repeatedly, leading to unnecessarily long computation times~\cite{Liberti2012,Margot2010,PfetschRehn2019}.
Permutation symmetries arising from exchanging spheres or dimensions have been studied extensively and many sophisticated symmetry handling (SH) methods exist, e.g., \cite{DoornmalenHojny2024a,KaibelPfetsch2008,Margot2002,OstrowskiEtAl2011}.
For rotation symmetries, however, only simple variable fixings~\cite{HoarePal1971} have been suggested or finite rotation subgroups have been tried to be handled~\cite{Doornmalen2026} to the best of our knowledge.

For point arrangement problems, rotations~$\rho$ in~$\R^\dim$ act on an arrangement~$X = (X^1,\dots,X^n) \in \R^{n \times \dim}$ as~$\rho(X) = (\rho(X^1),\dots,\rho(X^n))$.
To handle rotation symmetries, \cite{HoarePal1971} suggests fixing~$X_{i,j} = 0$ for all~$i \in [L]$, $j \in \{i+1,\dots,\dim\}$, where~$L = \min\{n,\dim\}$.
For translation invariant problems, they also suggest fixing~$X_{1,1} = 0$, but this is not applicable when only considering rotations.
In the following, we show that these fixings are the strongest possible variable fixings that can be derived from so-called Givens rotations.

A Givens rotation~$\rho$ is defined by coordinates~$j,j' \in [\dim]$, $j < j'$, and~$\alpha \in [0,2\pi)$.
Then, $\rho$ acts like an~$\alpha$-rotation in the $j$-$j'$-plane and keeping all remaining dimensions invariant.
From~\cite[Thm.~8]{DoornmalenHojny2024a}, it follows that all symmetries of Givens rotations are handled when enforcing the lexicographic relation
\begin{align}
  \label{eq:lexrotation}
  &(X_{1,j}, X_{1,j'}) \lexgeq\\
  &
  (\cos(\alpha) X_{1,j} - \sin(\alpha) X_{1,j'},
  \sin(\alpha) X_{1,j} + \cos(\alpha) X_{1,j'}).\notag
\end{align}
For each~$\alpha \in [0,2\pi)$, the first coordinate implies the inequality~$(1 - \cos(\alpha))X_{1,j} + \sin(\alpha) X_{1,j'} \geq 0$.
Basic calculus techniques show
$\min_{\alpha \in [0,2\pi)} (1 - \cos(\alpha))X_{1,j} + \sin(\alpha) X_{1,j'} = X_{1,j} - \sqrt{X_{1,j}^2 + X_{1,j'}^2}$.
Thus, all feasible solutions satisfy $X_{1,j'} = 0$ and~$X_{1,j} \geq 0$.
Consequently, Givens rotations and~\cite[Thm.~8]{DoornmalenHojny2024a} imply~$X_{1,1} \geq 0$ and~$X_{1,j'} = 0$ for all~$j' > 1$.

Iteratively applying Givens rotations to rows~$i = 2,\dots,L$ for Givens rotations keeping dimensions~$1,\dots,i-1$ invariant shows $X_{i,j} = 0$ for all~$i \in [L]$, $j \in \{i+1,\dots,\dim\}$ as claimed.
Moreover, it is easy to see that any solution adhering to these fixings satisfies~\eqref{eq:lexrotation} for all Givens rotations.
That is, despite their simplicity, no more symmetry reductions can be derived from Givens rotations.

In our implementation, we compare the effect of (not) handling rotation symmetries.
When rotation symmetries are handled, we apply the above mentioned variable fixings.
Additionally, we use orbitopal reduction~\cite{DoornmalenHojny2024a} to lexicographically sort rows~$L+1,\dots,n$ of~$X$ and add the inequalities~$X_{L+1,1} \geq X_{L+2,1} \geq \dots \geq X_{n,1}$.
This is possible, because the variable fixings and bound reductions do not act on points~$L+1,\dots,n$, i.e., they are still interchangeable and we can assume w.l.o.g.\ that they are sorted.
When rotation symmetries are not handled, we sort the rows and columns of~$X$ by orbitopal reduction and analogous inequalities as before.
Moreover, since the sphere packing and kissing number problem admit reflection symmetries, we enforce that some entries of~$X$ are nonnegative, see~\cite[Prop.~2]{Hojny2025}.
\section{Numerical experience}
\label{sec:num}
To evaluate the impact of the different algorithms for finding domain reductions based on \minDCs, we have implemented all techniques in the global MINLP solver \scip.
The experiments have been conducted for three different test sets:
general benchmark instances from the library MINLPLib~\cite{MINLPLIB} for which our implementation could automatically detect \minDCs (MINLP), instances of the obnoxious facility location (OFL) problem~\cite{Drezner2019} from EuclidLib~\cite{KuznetsovSahinidis2026}\footnote{https://github.com/anatoliy-kuznetsov/EuclidLib}, and variants of the kissing number problem and the problem of packing identical spheres into a bigger sphere (GEOM).
The dimension~$\dim$ of \minDCs for the MINLP test set is~$\dim \in \{1,\dots,5\}$, for OFL~$\dim = 2$, and for GEOM~$d \in \{2,3\}$.
For GEOM we also test the impact of handling rotation symmetries via the proposed variable fixings.

\paragraph{Hardware and software specifications.}
All experiments have been conducted on a Linux Cluster with Intel Xeon E5-1620 v4 \SI{3.5}{\giga\hertz} quad core processors and \SI{32}{\giga\byte} memory.
The code was executed single-threaded with a time limit of \SI{2}{\hour}.
We use \scip~10.0.0~\cite{SCIP10} as branch-and-bound framework; LP relaxations are solved using \solver{SoPlex}~8.0.0 and nonlinear problems are solved by \solver{Ipopt}~3.14.20~\cite{WachterBiegler2006}.
Symmetries are detected by \solver{sassy}~1.1~\cite{AndersEtAl2023} and \solver{Nauty}~2.8.8~\cite{Nauty}.
We use \scip's default parameters except for the optimality gap, which we set to~\SI{0.5}{\percent}.

\paragraph{Results for MINLP and OFL.}
Due to space restrictions, we discuss the performance of the different algorithms mainly using performance profiles~\cite{DolanMore2002} and refer the reader to an online supplement~\cite{supplement} for detailed numerical results and our code, which is also available on GitHub\footnote{https://github.com/christopherhojny/distance-constraints}.
A performance profile compares a method or setting to the best method for each problem.
The horizontal axis shows a performance ratio~$\tau \geq 1$, and the vertical axis the fraction of problems for which the method is within factor~$\tau$ of the best (competitive) method.
A higher profile indicates better overall performance, and the value for~$\tau = 1$ shows how often a method performed best.

For all test sets, we compare the performance of the different settings based on running time.
The profiles are based on all instances of the respective test set that could be solved by at least one setting within the time limit.
In preliminary experiments, we observed that separating the cutting planes of Sec.~\ref{sec:algOneCut} does not improve performance and can degrade performance if they are separated frequently.
We therefore disabled these cutting planes for all experiments.
Figs.~\ref{fig:minlpT} and~\ref{fig:oflT} show the results for the MINLP and OFL test sets, respectively, comparing our different domain reduction techniques.
The \emph{default} setting refers to \scip's default configuration without any of our techniques; \emph{heur\_0\_pair\_0} and \emph{heur\_1\_pair\_0} refer to the techniques from Secs.~\ref{sec:locatelli} and~\ref{sec:algOne}, respectively; and \emph{heur\_0\_pair\_1} and \emph{heur\_1\_pair\_1} refer to the algorithm of Sec.~\ref{sec:algTwo} with the geometrically inspired techniques and bisection, respectively.

For MINLP, all techniques for handling \minDCs substantially improve \scip's running time, with the settings involving \emph{heur\_0} performing better in general.
For OFL, however, there is a clear distinction between \emph{pair\_0} and \emph{pair\_1}.
While \emph{pair\_0} outperforms the \emph{default} setting, \emph{pair\_1} performs worse on most instances.

Overall, \emph{heur\_0\_pair\_0} is the most performant configuration.
Despite its simplicity, the method proposed in Sec.~\ref{sec:algOne} solves for both MINLP and OFL almost as many instances as the \emph{heur\_0\_pair\_0} setting.
In particular, for OFL, \emph{heur\_0\_pair\_0} and \emph{heur\_1\_pair\_0} are on approx.~\SI{60}{\percent} and~\SI{40}{\percent} of the instances the fastest settings, whereas the remaining settings yield hardly the fastest running time.
This shows that both \emph{heur\_0\_pair\_0} and our method \emph{heur\_1\_pair\_0} are complementary.
Especially in the presence of many \minDCs, it is advantageous to use more inexpensive algorithms for handling \minDCs that can be invoked frequently during the spatial branch-and-bound process to benefit from locally updated variable bounds.

\begin{figure}[t]
  \centering
  \subfigure[MINLP\label{fig:minlpT}]{%
    \includegraphics[scale=0.45]{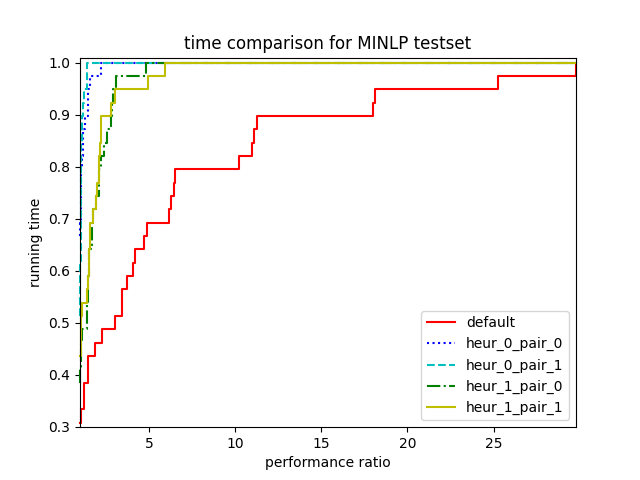}
  }
  \subfigure[OFL\label{fig:oflT}]{%
    \includegraphics[scale=0.45]{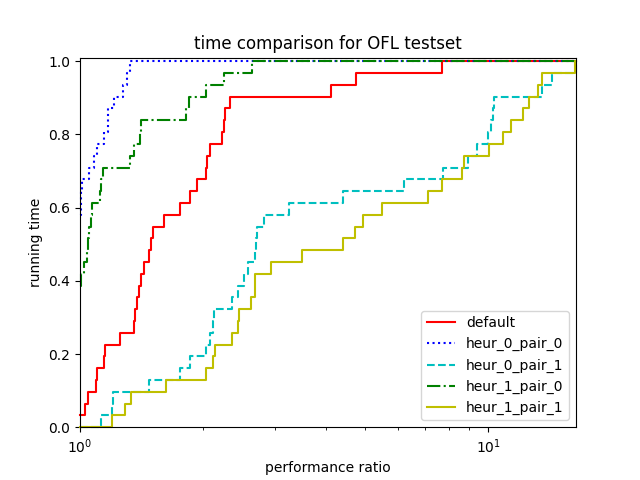}
  }
  \caption{Performance profiles for running times on the MINLP and OFL test set.}
  \label{fig:minlp}
\end{figure}

\paragraph{Results for GEOM.}
The GEOM test set consists of 2- and 3-dimensional versions of the following problems.
Finding the maximum radius such that~$n$ identical spheres fit into a sphere of radius~1, and finding the maximum radius such that the kissing number problem for~$n$ spheres has a solution.
For the former, we adapt the MINLP model~(1) from~\cite{Khajavirad2024} for packing spheres into a box, where~$n \in \{4,\dots,20\}$ when~$\dim = 2$ and~$n \in \{4,\dots,10\}$ when~$\dim = 3$.
For the kissing number problem we use Models~(1)--(5) and~(6) from~\cite{KucherenkoEtAl2007}.
The first model uses~$n \in \{4,\dots,20\}$ when~$\dim = 2$ and~$n \in \{4,\dots,10\}$ when~$\dim = 3$; the second model uses~$n \in \{4,\dots,20\}$ when~$\dim = 2$ and~$n \in \{4,\dots,16\}$ when~$\dim = 3$.

Fig.~\ref{fig:rot1} compares the results of the different settings when rotation symmetries are handled.
As for the other test sets, handling \minDCs greatly improves \scip's performance.
In general, the settings with \emph{heur\_0} dominate those with \emph{heur\_1}.
In contrast to the MINLP test set, however, one cannot benefit from handling pairs of \minDCs.

Fig.~\ref{fig:rotcomp} illustrates the impact of handling rotation
symmetries, where settings whose names start with ``rot\_'' incorporate handling rotation symmetries.
The plot includes all instances that were solved by at least one of the tested settings.
While handling rotation symmetries clearly improves the performance of the default setting, a difference between easy and difficult instances can be observed for the remaining settings.
For the easier instances, symmetry handling yields little benefit and may even introduce some overhead.
In contrast, for the more challenging instances, handling rotation symmetries is crucial and leads to substantially better overall performance.

\begin{figure}[t]
  \centering
  \subfigure[With handling of rotation symmetries.\label{fig:rot1}]{%
    \includegraphics[scale=0.45]{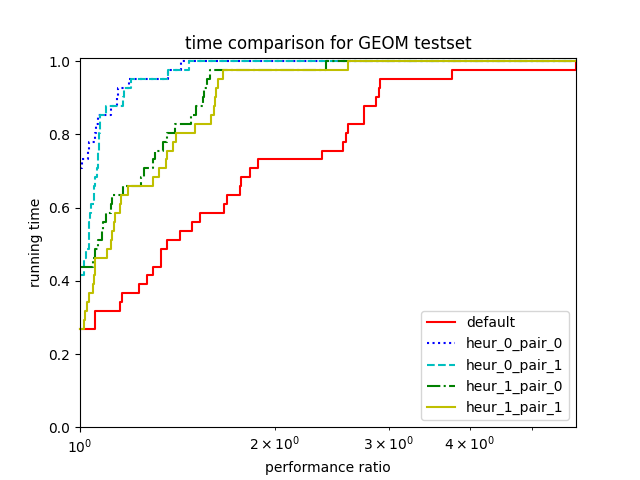}
  }
  \subfigure[Impact of handling rotation symmetries.\label{fig:rotcomp}]{%
    \includegraphics[scale=0.45]{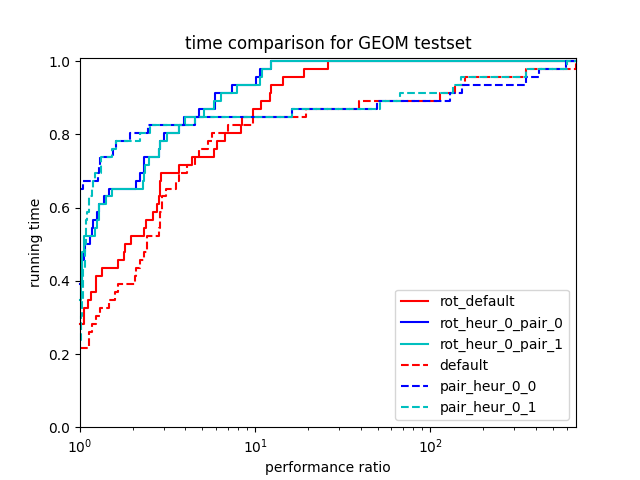}
  }
  \caption{Performance profiles for GEOM test set.}\label{fig:geom}
\end{figure}

Finally, we comment on an observation from~\cite{LocatelliRaber2002}: a domain reduction obtained for one box may trigger further reductions for other boxes, a phenomenon referred to as the \emph{wave effect}.
\scip exploits this effect automatically by repeatedly calling our algorithms until no further reductions are found or a limit on the number of consecutive calls is reached (1000 at the root node and 100 elsewhere in the branch-and-bound tree).
Consequently, all results reported above benefit from the wave effect.

For the GEOM test set, we investigated the impact of disabling these iterative calls.
Interestingly, not exploiting the wave effect led to improved overall running times, particularly for the more computationally expensive settings that include \emph{pair\_1}.
Due to space restrictions, we refer the reader to the online supplement~\cite{supplement} for more detailed figures.
This suggests that, in this setting, it is preferable to progress more quickly through the branch-and-bound tree rather than investing additional effort in obtaining the strongest possible local variable bounds.

\paragraph{Conclusion.}
Our computational study demonstrates that efficiently handling \minDCs is crucial for achieving good performance across all test sets.
Simple and inexpensive domain reduction techniques, in particular the heuristic from Sec.~\ref{sec:algOne}, consistently perform very well.
The latter, however, is slower than the methods from Sec.~\ref{sec:locatelli}.
We stress, however, that the methods from Sec.~\ref{sec:locatelli} are substantially more complex than the methods from Sec.~\ref{sec:algOne};
in particular, their running time depends exponentially on~$\dim$.
The running time of our methods, however, only depends linearly on~$\dim$.
They therefore provide a good trade off between effectiveness and running time, which might be particularly helpful for problems in higher dimensions.

For instances with many \minDCs, such as the OFL test set, lightweight reductions that are not based on pairs of \minDCs are especially beneficial, as they can be applied frequently during the spatial branch-and-bound process without incurring excessive overhead.
Handling rotation symmetries further improves performance on the GEOM instances on harder instances

\paragraph{Acknowledgments.}
We thank Stefan Vigerske for reviewing parts of our code as well as two reviewers for their helpful comments.
LL was partly sponsored by the ANR "Evariste" project ANR-24-CE23-1621.

\bibliographystyle{abbrv}
\bibliography{geometric_problems}

\end{document}